\documentclass[psamsfonts]{amsart}
\usepackage{amssymb,amsfonts}
\usepackage[all,arc]{xy}
\usepackage{enumerate}
\usepackage{mathrsfs}
\usepackage{ stmaryrd }

\newcounter{qcounter}

\newenvironment{abcs}{
  \begin{list}
  {(\arabic{qcounter}) ~}
  {
    \usecounter{qcounter}
    \setlength{\itemsep}{0in}         
    \setlength{\topsep}{0in}
    \setlength{\partopsep}{0in}
    \setlength{\parsep}{0.05in}
    \setlength{\leftmargin}{0.285in}   
    \setlength{\rightmargin}{0in}
    \setlength{\itemindent}{0in}
    \setlength{\labelsep}{0in}
    \setlength{\labelwidth}{\leftmargin}
    \advance \labelwidth by-\labelsep
  }
}{
  \end{list}
}

\newcommand{\cyc}{\mathbf{1} \boldsymbol{-} \boldsymbol{\zeta}}

\newcommand{\Z}{\mathbb{Z}}
\newcommand{\Q}{\mathbb{Q}}

\newcommand{\Zp}{\Z/{p\Z}}

\newcommand{\QpZp}{\Q_p/{\Z_p}}

\newcommand{\fH}{\mathfrak{H}}
\newcommand{\h}{\mathfrak{h}}
\newcommand{\fQ}{\mathcal{Q}}
\newcommand{\fP}{\mathcal{P}}
\newcommand{\fR}{\mathcal{R}}
\newcommand{\fO}{\mathcal{O}}

\newcommand{\tH}{\tilde{H}}
\newcommand{\Lam}{\Lambda}
\newcommand{\I}{\mathcal{I}}

\newcommand{\U}{\Upsilon}
\newcommand{\Th}{\Theta}
\newcommand{\bnu}{\overline{\nu}}
\newcommand{\chii}{{\chi^{-1}}}
\newcommand{\cL}{\mathcal{L}}

\newcommand{\Xc}{X^-_\chii}

\newcommand{\Hom}{\text{Hom}}
\newcommand{\Gal}{\text{Gal}}
\newcommand{\Frac}{\text{Frac}}
\newcommand{\Aut}{\text{Aut}}

\newcommand{\Ext}{\text{Ext}}

\newcommand{\HDM}{\tH_{DM}}

\newcommand{\X}{\mathfrak{X}}

\newcommand{\zinf}{\{0,\infty\}}

\newtheorem{thm}{Theorem}[section]
\newtheorem{cor}[thm]{Corollary}
\newtheorem{prop}[thm]{Proposition}
\newtheorem{lem}[thm]{Lemma}

\theoremstyle{definition}

\theoremstyle{remark}
\newtheorem{rem}[thm]{Remark}
\newtheorem{rems}[thm]{Remarks}

\makeatletter
\let\c@equation\c@thm
\makeatother
\numberwithin{equation}{section}

\title{Hecke algebras associated to $\Lam$-adic modular forms}

\author{Preston Wake}

\date{}

\begin{document}

\maketitle

\begin{abstract}
We show that if an Eisenstein component of the $p$-adic Hecke algebra associated to modular forms is Gorenstein, then it is necessary that the plus-part of a certain ideal class group is trivial. We also show that this condition is sufficient whenever a conjecture of Sharifi holds.
\end{abstract}

\section{Introduction}

In this paper, we study whether the Eisenstein component of the $p$-adic Hecke algebra associated to modular forms is Gorenstein, and how this relates to the theory of cyclotomic fields. We will first prepare some notation in order to state our results. See section \ref{notation remarks} for some remarks on the notation.

\subsection{Notation}\label{notation} 

Let $p \ge 5$ be a prime and $N$ an integer such that $p \nmid \varphi(N)$ and $p \nmid N$. Let $\theta: (\Z/{Np \Z})^\times \to \overline{\Q}_p^\times$ be an even character and let $\chi=\omega^{-1} \theta$, where $\omega : (\Z/{Np \Z})^\times \to (\Z/p\Z)^\times \to \Z_p^\times$ denotes the Teichm\"{u}ller character. We assume that $\theta$ satisfies the same conditions as in \cite{sharifi} and \cite{kato-fukaya} -- namely that 1) $\theta$ is primitive, 2) if $\chi |_{(\Zp)^\times}=1$, then $\chi |_{(\Z/{N\Z})^\times}(p) \ne 1$,  3) if $N=1$, then $\theta \ne \omega^2$.

If $\phi: G \to \overline{\Q}_p^\times$ is a character of a group $G$, let $\Z_p[\phi]$ denote the $\Z_p$-algebra generated by the values of $\phi$, on which $G$ acts through $\phi$. If $M$ is a $\Z_p[G]$-module, denote by $M_\phi$ the $\phi$-eigenspace:
$$
M_\phi = M \otimes_{\Z_p[G]} \Z_p[\phi].
$$

Let $$
H'=\varprojlim H^1(\overline{X}_1(Np^r),\Z_p)^{ord}_{\theta}$$
and
$$
\tH'= \varprojlim H^1(\overline{Y}_1(Np^r),\Z_p)^{ord}_{\theta},$$
where $ord$ denotes the ordinary part for the dual Hecke operator $T^*(p)$, and the subscript refers to the eigenspace for the diamond operators.

Let $\h'$ (resp. $\fH'$) be the algebra of dual Hecke operators acting on $H'$ (resp. $\tH'$). Let $I$ (resp. $\I$) be the Eisenstein ideal of $\h'$ (resp. $\fH'$). Let $\fH$ denote the {\em Eisenstein component} $\fH = \fH'_\mathfrak{m}$ the localization at the unique  maximal ideal $\mathfrak{m}$ containing $\I$. We can define the Eisenstein component $\h$ of $\h'$ analogously. Let $\tH=\tH' \otimes_{\fH'} \fH$ and $H=H' \otimes_{\h'} \h$, the Eisenstein components.

Let $G_\Q = \Gal(\overline{\Q}/\Q)$. For a $G_\Q$-module $M$, let $M^+$ and $M^-$ for the eigenspaces of complex conjugation.

Let $\Q_\infty=\Q(\zeta_{Np^\infty})$; let $M$ be the maximal abelian $p$-extension of $\Q_\infty$ unramified outside $Np$ and let $L$ be  the maximal abelian $p$-extension of $\Q_\infty$ unramified everywhere. Let $\X=\Gal(M/\Q_\infty)$ and $X=\Gal(L/\Q_\infty)$.

Let $\Z_{p,N}^\times = \Z_p^\times \times (\Z/N\Z)^\times$ and $\Lam=\Z_p[[\Z_{p,N}^\times]]_\theta$, the Iwasawa algebra. We use the identification $\Lam \cong \Z_p[[\Gal(\Q_\infty/\Q)]]_\theta$ to give an action of $\Lam$ on $\Gal(\Q_\infty/\Q)$-modules. Let $\xi \in \Lam$ be a characteristic power series of $X_\chi^-$. For more on this, see \cite{washington}. Let $\iota: \Z_p[[\Z_{p,N}^\times]] \to \Z_p[[\Z_{p,N}^\times]]$ by the involution given by $c \mapsto c^{-1}$ on $\Z_{p,N}^\times$. Let $\tau : \Z_p[[\Z_{p,N}^\times]] \to \Z_p[[\Z_{p,N}^\times]]$ be the morphism induced by $\langle c \rangle \mapsto c\langle c \rangle$ for $c \in \Z_{p,N}^\times$. For a $\Z_p[[\Z^\times_{p,N}]]$-module $M$, we let $M^\#$ (resp. $M(r)$) be the same abelian group with $\Z_p[[\Z^\times_{p,N}]]$-action changed by $\iota$ (resp. $\tau^r$).

\subsection{Statement of Results} Many authors have studied the Gorenstein property of Hecke algebras, and its relationship to arithmetic (cf. \cite{kurihara}, for example).  Ohta has the following theorem.

\begin{thm}{\em (Ohta, \cite{comp2})}\label{ohtaresult}
Suppose that $\X^+_\theta = 0$. Then $\fH$ is Gorenstein.
\end{thm}

Skinner and Wiles have obtained similar results using different methods \cite{skinner-wiles}. This is a sufficient condition for $\fH$ to be Gorenstein. The main result of this paper gives a necessary condition that is conjecturally also sufficient.

\begin{thm}
\label{main}
Suppose that $X^-_\chi \neq 0$. Then
\abcs
\item If $\fH$ is Gorenstein, then $X^+_\theta=0$. 
\item If Sharifi's conjecture is true, then $X^+_\theta=0$ implies $\fH$ is Gorenstein.
\endabcs
\end{thm}

\begin{rems}\label{mainrem}
By Sharifi's conjecture, we mean the statement that the map $\U$, described in section \ref{up and theta} below, is a surjection. In fact, it follows from \cite[Conjectures 4.12, 5.2 and 5.4]{sharifi} that $\U$ is an isomorphism. For partial results on Sharifi's conjecture, see \cite{kato-fukaya}.

 Note that $X^+_\theta$ is a quotient of $\X^+_\theta$, which appears in Ohta's result. If Sharifi's conjecture is true, then this is a stronger result.

 The assumption $X^-_\chi \ne 0$ is just to make the result interesting. Note that $X^-_\chi = 0$ if and only if $\xi \in \Lam^\times$. However, $\h/I \cong \Lam/\xi$ \cite[Corollary A.2.4]{cong}, so $\xi \in \Lam^\times$ if and only if $\h=I=0$. Since $\fH/\I \cong \Lam$, if $I=0$ then $\fH=\Lam$. Thus the question of whether $\fH$ is Gorenstein is only interesting if $X^-_\chi \ne 0$.

\end{rems}

From Ohta's and Skinner-Wiles's results, one may hope that $\fH$ is always Gorenstein, but our result shows that this is not true.

\begin{cor}\label{maincor}
The ring $\fH$ is not always Gorenstein.
\end{cor}
\begin{proof}
We need only find $p$, $N$ and $\theta$ satisfying the assumptions of \ref{notation} and such that $X^+_\theta \ne 0$ and $X^-_\chi \neq 0$. 
To find an example, we first note that, by Iwasawa theory (cf. section 3), the condition $X^+_\theta \ne 0$ is implied by the condition that $\X_\theta^+$ is not cyclic as a $\Lam$-module. By the theory of Iwasawa adjoints (cf. section 4), $\X_\theta^+$ is not cyclic if $X_\chii^-$ is not cyclic. 

One way to search for examples is to find an odd primitive character $\chi$ of conductor $pN$ (with $p \nmid \varphi(N)$) and of order two such that $X^-_\chii=X^-_\chi$ is not cyclic (and hence not zero).

To find such a character, we look at tables of class groups of imaginary quadractic fields (\cite{tables}) and find that, for $p=5$ and $N=350 267$, the $p$-Sylow subgroup of the class group of $\Q(\sqrt{-Np})$ is not cyclic. Let $\chi: G_\Q \to \{\pm 1\} $ be the character associated to $\Q(\sqrt{-Np})$; since $\Q(\sqrt{-Np})$ embeds into $\Q(\zeta_{Np})$, this factors through $(\Z/Np\Z)^\times$. This is the required $\chi$.

\end{proof}

We also have applications to the theory of cyclotomic fields, as in \cite{comp2}.

\begin{cor}
Assume that the order of $\chi$ is $2$ and that $X^-_{\chi} \ne 0$. If $\fH$ is Gorenstein, then the following are true:
\begin{abcs}
\item\label{cyc} $X^-_\chi$ is cyclic.
\item $\h$ is a complete intersection.
\item $I$ is principal.
\item $\I$ is principal.
\end{abcs}
In particular, this holds when $X^+_\theta=0$ if Sharifi's conjecture is true.
\end{cor}
\begin{proof}
By part 1 of Theorem \ref{main} we have that $X^+_\theta=0$. It follows from $X^+_\theta=0$  that $X^-_\chii$ is cyclic as in the proof of Corollary \ref{maincor}. Now, since $\chi$ is order 2, we have that $X^-_\chi=X^-_\chii$, and thus that item (1) is true. Since $\fH$ is Gorenstein, we have by \cite[Corollary 4.2.13]{comp2} that (1) - (4) are equivalent, and thus that (1) - (4) are true.

The last statement is immediate from part 2 of Theorem \ref{main}.
\end{proof}

\begin{rem}
This corollary has a version without the assumption that $\chi$ is order $2$: one has to assume that $X^-_\chi$ and $X^-_\chii$ are both non-zero, and that both $\fH$ and the analog of $\fH$ for $\theta'=\omega \chii$ are Gorenstein. 
\end{rem}

\subsection{Outline of the proof}

The idea of the proof comes from Kato and Fukaya's work on Sharifi's conjectures \cite{kato-fukaya}. They consider the Drinfeld-Manin modification $\tH_{DM}=\h \otimes_\fH \tH$ of $\tH$ and some subquotients: 
$$
\fR=\tH_{DM}/H, \ \fP=H^-/{IH^-}, \ \fQ=(H/IH)/\fP.
$$
There are canonical isomorphisms of $\Lambda$-modules
$$
\fR \cong \Lambda/\xi(-1) \ , \ \fQ \cong (\Lambda/\xi)^\#.
$$
Using Otha's theory of $\Lam$-adic Eichler-Shimura cohomology, one can show that $\tH^- \cong \Hom_\Lam(\fH,\Lam)$, and thus that $\fH$ is Gorenstein if and only if $\tH^-$ is free of rank one as an $\fH$-module.

We construct a commutative diagram:

\begin{equation*}\tag{*}
\xymatrix{
\X^-_\chii \otimes X_\chi^- \ar[d]^{\nu' \otimes \U} \ar[r]^-{\Th \otimes \U} & \Hom_{\h}(\fR,\fQ) \otimes \Hom_\h (\fQ,\fP) \ar[d] \\
 (\Lambda/\xi)^\#(1) \otimes \Hom_\h (\fQ,\fP)  \ar[r] & \fP(1).
}
\end{equation*}

The unlabeled maps are the natural ones (using the above canonical isomorphisms for $\fR$ and $\fQ$). 

In section 2, we describe the top horizontal map, and we show that $\fH$ is Gorenstein if and only if the clockwise map $\X_\chii^- \otimes X_\chi^- \to \fP(1)$ is surjective. In section 3, we describe the map $\nu'$ and show that the diagram is commutative.  It is known that if $\fH$ is Gorenstein, then $\U$ is an isomorphism. Since $X^-_\chi \ne 0$, we get that if $\fH$ is Gorenstein, then $\nu'$ is surjective. In section 4, we recall the theory of Iwasawa adjoints, and use this to show that $\nu'$ is surjective if and only if $X_\theta^+=0$. This completes the proof of the first statement.

In Sharifi's paper \cite{sharifi}, he formulates conjectures that imply that $\U$ is an isomorphism even without the hypothesis that $\fH$ is Gorenstein (cf. \cite{kato-fukaya}). If $\U$ is surjective and $X^+_\theta=0$, then $\nu'$ is also surjective, so the clockwise map $\X_\chii^- \otimes X_\chi^- \to \fP(1)$ is surjective, and so $\fH$ is Gorenstein.

\subsection{Remarks on the Notation}\label{notation remarks} There are many small choices of notational convention in this area of study, and it seems that every author has a different convention. The major result that is used in this paper comes from \cite{kato-fukaya}, so we align to their notation. 

We use the standard model for $X_1(M)$: as the moduli space for $(E, \Z/M\Z \subset E)$ elliptic curves with a subgroup isomorphic to $\Z/M\Z$. Sharifi and Ohta use the model that is a moduli space for $(E, \Z/M\Z(1) \subset E)$, which affects the Galois action (cf. Proposition \ref{gal}). In the notation of \cite{kato-fukaya}, we use the model $X_1(M)$, whereas Ohta and Sharifi use the model $X_1'(M)$.

 We use the algebra of dual Hecke operators $T^*(n)$, but the eigenspace with respect to the usual diamond operators $\langle a \rangle$. In effect, this means that our Hecke algebra $\fH$ would be denoted by $\fH_\mathfrak{m}^{\langle \chii \rangle}$ in \cite{sharifi}. We hope that, with this dictionary, the reader may now easily translate between the various papers mentioned herewith.

Finally, throughout the paper, we use notation such as $\X_\theta^+$, which is redundant since $\X_\theta^+=\X_\theta$. However, we find it helpful, since the plus-minus parts often determine the most important behaviour, and it can be difficult to remember the parity of the characters. We hope that this does not cause any confusion.

\subsection{Acknowledgments} The author would like to thank Professor Kazuya Kato for suggesting this problem, and for his constant encouragement and patient advice during the preparation of this paper. He would also like to thank Matthew Emerton, Takako Fukaya, Madhav Nori, Matsami Ohta, and especially Romyar Sharifi  for helpful comments and discussions about this work. He also thanks the anonymous referee for his diligent reading and his many constructive comments.

\section{Galois Actions}

In this section, we recall some results on $H/{IH}$. These results are mainly due to many authors -- Ohta, Mazur-Wiles, Sharifi, Fukaya-Kato. See \cite[sections 6.2-6.3]{kato-fukaya} for an excellent exposition, including proofs. The purpose of this section is to construct the clockwise map in the commutative diagram (*), and show that it is surjective if $\fH$ is Gorenstein.

\subsection{}\label{bc} Ohta has shown that there is a decomposition, using the action of $\Gal(\overline{\Q}_p/\Q_p)$, of $H$ into a direct sum of a dualizing $\h$-module and a free $\h$-module of rank 1 (\cite[section 4.2]{comp2}, for example). By \cite[Theorem 4.3]{sharifi}, Ohta's decomposition and the decomposition into plus and minus parts are the same. See \cite[Proposition 6.3.5]{kato-fukaya} for a simple and self-contained proof of this. This compatibility will be used without further comment in the remainder.

In particular, $\tilde{H}^+=H^+$ is a free $\h$-module of rank 1, $H^- \cong \Hom_\Lam(\h,\Lam)$, and $\tilde{H}^- \cong \Hom_\Lam(\fH,\Lam)$.

Let $\cL=\Frac(\Lam)$. For a $\Lam$-module M, we let $M_\cL=M\otimes_\Lam \cL$. It follows from the above that
$H^-_\cL$ and $(\HDM^-)_\cL$ are free $\h_\cL$-modules of rank 1.

\subsection{}\label{R and 0-infinity} By Theorem 1.5.5 of \cite{cong}, there is an exact sequence
$$\xymatrix{
0  \ar[r] & H \ar[r] & \tilde{H} \ar[r]^\delta & \Lambda \ar[r] & 0 ,
}
$$
where $\delta$ is the boundary at $0$-cusps (cf. \cite[6.2.5]{kato-fukaya}). Let $\zinf' \in \tilde{H}$ be the class corresponding via Poincare duality to the relative homology class of the path from $0$ to $\infty$, so that $\delta(\zinf')=1$. This implies that $\zinf'$ is a generator for $\tilde{H}/H \cong \Lambda \cong \fH/\I$. Denote by $\zinf$ the image of $\zinf'$ in $\fR$. We have that $\zinf$ is a generator of $\fR \cong \Lambda/\xi \cong \h/I$.

Note that $\fR$ is a torsion $\Lam$-module, so the inclusion $H^- \to \HDM^-$ induces an isomorphism after tensoring with $\h_\cL$.

\subsection{}\label{up and theta} The action of $G_\Q$ on $\fR, \ \fP, \ \fQ$ is known, and is summarized by the following

\begin{prop}\label{gal}
$\fP$ is a $G_\Q$-submodule of $H/{IH}$; $\sigma \in G_\Q$ acts as $\kappa(\sigma)^{-1}$ on $\fP$ and $\fR$, where $\kappa$ is the $p$-adic cyclotomic character, and as $\langle \sigma \rangle^{-1}$ on $\fQ$. 
\end{prop}

Sharifi has defined a homomorphism
$$
\U: X^-_\chi \to \Hom_\h(\fQ,\fP)
$$
where, for $\sigma \in X$, $\U(\sigma)$ is induced by the map  $x \mapsto (\sigma-1)x$, for $x \in H/{IH}$: since the action of $X$ on $\fP$ is trivial, this factors through $\fQ$; since the action of $X$ on $\fQ$ is trivial, we have $(\sigma -1)x \in \fP$ for all $x \in H/{IH}$. This map is clearly Galois equivariant, and so factors through $X^-_\chi$.

We can similarly define a homomorphism $$\Th: \X^-_\chii \to \Hom_\h(\fR,\fQ)$$ where, for $\tau \in \X$, $\Th(\tau)$ is induced by the map $r \mapsto ((\tau-1)r \text{ mod } I) \text{ mod } \fP$, for $r \in \tH_{DM}$: since $\X$ acts trivially on $\fQ$, we see that this factors through $\fR$; since $\X$ acts trivially on $\fR$, we see that $(\tau-1)r \in H$ for any $r \in \tH_{DM}$. This map is clearly Galois equivariant, and so factors through $\X^-_\chii$.

We have a natural map $\Hom_\h (\fR,\fQ) \otimes \Hom_\h (\fQ, \fP) \to \Hom_\h (\fR,\fP)$, given by composition. There is then a natural map $\Hom_\h (\fR, \fP) \to \fP(1)$ given by evaluation at $\zinf$, the canonical basis of $\fR$.

We obtain the clockwise map $\Phi : \X_\chii^- \otimes X_\chi^- \to \fP(1)$ in the commutative diagram (*): it is given by $\Phi(\tau \otimes \sigma)=\U(\sigma) \circ \Th(\tau) ({\zinf})$. 

\subsection{}\label{surjectivity of phi} The next result determines the image of $\Phi$. 

\begin{prop}\label{image}
The image of $\Phi$ is $I \zinf$, where $I$ is the Eisenstein ideal.
\end{prop}
\begin{proof}
As in \cite[section 4.2]{comp2}, we consider the matrix decomposition of the Galois representation $\rho: G_\Q \to \Aut(\HDM)$ according to the direct sum $\HDM = \HDM^- \oplus H^+$:
$$\rho(\tau)=\left(
\begin{array}{ccc}
a(\tau) & b(\tau) \\
c(\tau) & d(\tau)
\end{array}
\right),
$$ 
where, for example, $c(\tau) \in \Hom(\HDM^-,H^+)$. 
Note that since $(\HDM^-)_\cL$ and $H^+_\cL$ are free of rank $1$ over $\h_\cL$, we fix bases so that $\rho(\tau)$ can be considered as a matrix in $GL_2(\h_\cL)$. Let $B$ be the $\h$-submodule of $\h_\cL$ generated by $\{b(\tau) \ | \ \tau \in \Gal(\overline{\Q}/\Q_\infty)\}$, and similarly $C$ by the $c(\tau)$. We then have $BC=I$ by the work of Ohta, as explained in \cite[Proposition 4.2]{sharifi} (recall section \ref{bc}).

For $\tau \in \X$, consider the map $\rho(\tau)-1: \HDM \to \HDM$. It is clear that, in matrix notation, the lower-left entry of $\rho(\tau)-1$ induces the map $\Theta(\tau)$. However, as subtracting the identity does not effect the off-diagonal entries, this lower-left entry is $c(\tau)$. Similarly, for $\sigma \in X$, $\U(\sigma)$ is induced by $b(\sigma)$. It is then clear that the image of $\Phi$ is $BC \zinf = I \zinf$.
\end{proof}

This leads to the main result of the section.

\begin{prop}\label{surj}
The map $\Phi$ is surjective if and only if $\fH$ is Gorenstein 
\end{prop}

The proof follows from the next two lemmas.

\begin{lem}
If $\fH$ is Gorenstein, then $\HDM^-$ is a free $\h$-module of rank $1$ with basis $\zinf$, and $H^-=I\zinf$. In particular, $\fP$ is generated by the image of $I\zinf$.
\end{lem}
\begin{proof}
The assumption $\fH$ is Gorenstein means that we have $\Hom_\Lam(\fH,\Lam) \cong \fH$ as $\fH$-modules. Recall from section \ref{bc} that $\tH^- \cong \Hom_\Lam(\fH,\Lam)$. We see that $\tH_{DM}^-$ is a free $\h$-module of rank 1. 

Now, since $\tH^-/H^-$ is generated by $\zinf$, as an $\fH$-module and $\tH^-$ is free of rank one, we see that $\tH^-$ is also generated by $\zinf$. 

We have the commutative diagram: 

\[\xymatrix{
0 \ar[r] & H^-   \ar[r] & \HDM^- \ar[r] & \fR \ar[r] & 0\\
0 \ar[r] & I  \ar[r] \ar[u] & \h \ar[r] \ar[u]  & \h/I \ar[r]  \ar[u] & 0\\} \]
where the centre and rightmost vertical arrows are the isomorphisms given by $1 \mapsto \zinf$. We see that $H^-=I \zinf$.
\end{proof}

\begin{lem}
If $\fP$ is generated by $I \zinf$, then $\fH$ is Gorenstein.
\end{lem}
\begin{proof}
By Nakayama's lemma, $H^-$ is generated by $I \zinf$. Then, in the diagram
\[\xymatrix{
0 \ar[r] & H^-   \ar[r] & \tH^- \ar[r] & \tH^-/H^- \ar[r] & 0\\
0 \ar[r] & \I  \ar[r] \ar[u] & \fH \ar[r] \ar[u]  & \fH/\I \ar[r]  \ar[u] & 0\\}, \]
where the centre and rightmost arrows are given by $1 \mapsto \zinf$, the leftmost and rightmost arrows are surjective. We see that the $\tH^- \cong \Hom_\Lam(\fH, \Lam)$ is generated by one element, and so $\fH$ is Gorenstein.
\end{proof}

\section{Relation to cyclotomic units}

In this section, we recall some definitions from the theory of cyclotomic fields. We apply these ideas to construct the map $\nu'$ in the diagram (*). We then show how a result of Fukaya and Kato implies that (*) is commutative.

\subsection{}\label{U and E} Let $\mathcal{U}=\varprojlim (\Z[\zeta_{Np^r}]\otimes \Z_p)^\times$ and $U$ the pro-$p$ part of $\mathcal{U}$. Let $\mathcal{E}$ be the closure of the global units in $\mathcal{U}$ and $E$ its pro-$p$ part. They are each a product of $g=\varphi(N)/d$ groups, where $d$ is the order of $p$ in $(\Z/N\Z)^\times$.

Recall that the Artin map induces an isomorphism $U/E \cong \Gal(M'/L)$, where $M'$ is the maximal abelian $p$-extention of $\Q_\infty$ unramified outside $p$, and $L$ is as in \ref{notation}. As a result, for any primitive character $\phi$ of $(\Z/Np\Z)^\times$, we get an exact sequence:
\begin{equation}\label{euxx}
0 \to E_\phi \to U_\phi \to \X_\phi \to X_\phi \to 0.
\end{equation}

\subsection{}\label{kummer pairing} The Hilbert symbol
$$\xymatrix{
 U \times U \ar[r]^-{( \ , \ )'_r} & \mu_{p^r}^g
}$$
is given on each component of $U$ by
$$
(a,b)'_r=\frac{\sigma_{b_r}(a_r^{1/p^r})}{a_r^{1/p^r}}
$$
where $\sigma_{b_r}$ is the image of $b_r$ under the Artin map. We define $( \ , \ )_r$  to be the composite of $( \ , \ )'_r$ with the product map $\mu_{p^r}^g \to \mu_{p^r}$.

We also have
$$\xymatrix{
E \times \X \ar[r]^-{[\ , \ ]_r} & \mu_{p^r}
}$$
given by
$$
[u,\sigma]_r=\frac{\sigma(u_r^{1/p^r})}{u_r^{1/p^r}}.
$$
with $u_r \in \Z_p[\zeta_{Np^r}]^\times$.

Choosing a compatible system of $p^r$-th roots of unity, we can and do identify $\mu_{p^r}$ with $\Z/{p^r\Z}$.

We have Kummer pairing maps
$$( \ , \ )_{Kum}: U \times U \to \Z_p[[\Z_{p,N}^\times]],  \ \ 
[ \ , \ ]_{Kum}: E \times \X \to \Z_p[[\Z_{p,N}^\times]]
$$
given by
$$(a,b)_{Kum}= \varprojlim \sum_{c \in (\Z/{Np^n \Z})^\times} (\sigma_c(a),b)_n \langle c^{-1} \rangle,$$ 
 $$[u,\sigma]_{Kum}= \varprojlim \sum_{c \in (\Z/{Np^n \Z})^\times} [\sigma_c(u),\sigma]_n \langle c^{-1} \rangle,$$ 
where $\sigma_c$ is defined by
$$
\sigma_c(\zeta_{Np^n})=\zeta_{Np^n}^c.
$$

\begin{lem}\label{properties of kummer pairing}
For $a, b \in U$, $u \in E$, $\sigma \in \X$, and $c \in \Z_{p,N}^\times$ we have
\abcs 
\item $\langle c \rangle(a,b)_{Kum}=(\sigma_c a, b)_{Kum}$ and $\langle c \rangle [u,\sigma]_{Kum}=[\sigma_c u,\sigma]_{Kum}$.
\item $(b,a)_{Kum}=-\iota \tau (a,b)_{Kum}$ and $[u,\sigma_c \sigma]_{Kum}=\iota \tau \langle c \rangle \cdot [u,\sigma]_{Kum}$.
\item $(u,a)_{Kum}=[u,\sigma_a]_{Kum}$, where $\sigma_a$ is the image of $a$ under the Artin map.
\endabcs
\end{lem}
\begin{proof}
Follows from an easy computation, using well-known properties of the pairings $( \ , \ )_r$ and $[ \ , \ ]_r$.
\end{proof}

In other words, the pairings induce pairings of $\Z_p[[\Z^\times_{p,N}]]$-modules
$$
U \times U^\#(1) \to \Z_p[[\Z^\times_{p,N}]]
$$
and
$$
E \times \X^\#(1) \to \Z_p[[\Z^\times_{p,N}]].
$$

In particular, there is an induced map
$$
E^+_\theta \to \Hom(\X^-_\chii,\Lambda^\#(1)).
$$

\begin{lem}\label{non-degeneracy}

The induced pairing
$$
U^+_\theta \times (U^-_\chii)^\#(1) \to \Lambda
$$
is non-degenerate.
\end{lem}
\begin{proof}
In Theorem \ref{explicit res}, it is proven that this pairing is isomorphic to the product pairing on the integral domain $\Lambda$. (Note that the proof of Theorem \ref{explicit res} does not use this Lemma \ref{non-degeneracy}, and so no circular reasoning arises.)
\end{proof}

\subsection{} \label{commutivity}
Let $\cyc \in E^+_\theta$ be the image of $(1-\zeta_{Np^r})_r \in \varprojlim (\Z[\zeta_{Np^r}]\otimes \Z_p)$. We have a map
$$
\nu : \X^-_\chii \to \Lam^\#(1)
$$
given by $\nu(\sigma)=[\cyc,\sigma]_{Kum}$, the pairing with cyclotomic units. The map $\nu': \X^-_\chii \to (\Lam/\xi)^\#(1)$ in the diagram (*) is the composition of $\nu$ with the natural map $\Lam^\#(1) \to (\Lam/\xi)^\#(1)$. 

\begin{prop}\label{comm}
As maps
\[\xymatrix{
\X^-_\chii  \ar[r]^-{\nu'} & (\Lam/\xi)^\#(1) }\]
we have $\nu'=\Theta$.
\end{prop}

This proposition is exactly saying that the diagram (*) is commutative. This is a restatement of the following theorem of Fukaya and Kato (\cite[Theorem 9.6.3]{kato-fukaya}). 

\begin{thm}\label{ext}
Let $\mathcal{E}$ denote the natural extension of $\fR$ by $\fQ$, given by their structure as subquotients of $\HDM$. Then the class of $\mathcal{E}$ is given by cyclotomic units.
\end{thm}

To explain how Proposition \ref{comm} follows from this, we must explain what is meant by `is given by cyclotomic units'.

The modules $\fR$ and $\fQ$ are both free of rank one as $\Lam/\xi$-modules, with canonical bases. By Proposition \ref{gal}, we have canonical isomorphisms of $\Lam/{(\xi)} [G_\Q]$-modules:
$$
\fR = \Lam/\xi(-1), \ \fQ = (\Lam/\xi)^\#.
$$
Thus group of extension class of $\fR$ by $\fQ$ coincides with the group
$$
\Ext^1_{\Lam/{(\xi)} [G_\Q]}(\Lam/\xi(-1), (\Lam/\xi)^\#) = H^1(\Z[1/{Np}],  (\Lam/\xi)^\#(1)).
$$
As an element of this Galois cohomology, the class $\mathcal{E}$ is given by the map $\Th$ of section \ref{up and theta}. 

We have a natural map
$$\xymatrix{
H^1(\Z[1/Np], \Lam^\#(1)) \ar[r]^-j & H^1(\Z[1/Np], (\Lam/\xi)^\#(1)).}
$$
The map $\nu$ gives an element of $H^1(\Z[1/Np], \Lam^\#(1))$, and Theorem \ref{ext} says exactly that $j(\nu)=\nu'=\mathcal{E}$. Thus $\nu'=\Theta$, whence Proposition \ref{comm}.

\section{Iwasawa Adjunction}\label{Iwasawa Adjunction}

In this section, we recall the theory of Iwasawa adjoints in terms of Ext.

\subsection{}\label{x as h} We recall the interpretation of $\X$ and $X$ in terms of Galois cohomology. Let $G_r=\Gal(K/{\Q(\zeta_{Np^r})})$ and $G_\infty=\Gal(K/\Q_\infty)$, where $K$ is the maximal Galois extension of $\Q(\zeta_{Np})$ that is unramified outside $Np$, and $\Q_\infty$ is as in \ref{notation}. Let $H_\infty^q = \varprojlim_r H^q (G_r, \Z_p)$.
We have
$$
\X = H^1(G_\infty, \QpZp)^\vee,
$$
where $(-)^\vee$ denotes the Pontryagin dual, and, by \cite[Lemmas 2.1 and 4.11]{sharifi}, an isomorphism $$X(-1)_{\theta^{-1}} \cong (H_\infty^2)_{\theta^{-1}}.$$ (Recall our assumptions on $\theta$ (\ref{notation})).

\subsection{} For a $\Lam$-module $M$, denote by $\text{E}^i(M)$ the group $\Ext_\Lam^i(M,\Lam)$. There is a natural $\Lam$-module structure on $\text{E}^i(M)$ given by $(\lambda.f)(m)=f(\iota(\lambda) m)$ for $\lambda \in \Lam$, $f \in \text{E}^0(M)$ and $m \in M$, and by functoriality for higher Ext. The groups $\text{E}^i(M)$ are called the {\em (generalized) Iwasawa adjoints} of $M$.
 
\begin{lem}
The Kummer pairing gives an isomorphism of $\Lambda$-modules
$$
U^+_\theta \cong E^0(U^-_\chii)(1)
$$
\end{lem}
\begin{proof}
Noting our convention for the $\Lambda$-action on $\text{E}^0$, this is just a restatement of Lemma \ref{non-degeneracy}.
\end{proof}
 
Before we can compute more adjoints, we must first recall some generalities.

\begin{prop}\label{e and t}
For any finitely generated $\Lam$-module $M$, let $T_0(M)$ be the maximal finite submodule of $M$, and let $T_1(M)$ be the $\Lam$-torsion submodule of $M$. We have the following for any finitely generated $\Lam$-module $M$: 
\begin{abcs}
\item\label{lamtor} The module $\text{E}^1(M)$ is $\Lam$-torsion.
\item If $M$ is $\Lam$-torsion, then $T_0(\text{E}^1(M))=0$.
\item $\text{E}^1(M/T_0(M))=\text{E}^1(M)$, and $\text{E}^1(M)=0$ if and only if $M/T_0(M)$ is free.
\item $\text{E}^2(M) =T_0(M)^\vee$.
\item\label{EE} $\text{E}^1(\text{E}^1(M)) = T_1(M)/T_0(M)$.
\item $\text{E}^1(M)$ is finite if and only if $T_1(M)$ is finite.
\end{abcs}
\end{prop}
\begin{proof}
See \cite[section 5.4]{nsw}.
\end{proof}

This is relevant to our situation because of the following proposition.

\begin{prop}\label{spec}
There is a spectral sequence of finitely generated $\Lam$-modules
$$
\text{E}_2^{ \, p,q}=E^p(H^q(G_\infty,\QpZp)^\vee) \Rightarrow H^{p+q}_\infty
$$
\end{prop}
\begin{proof}
This is a special case of \cite[Theorem 1]{jannsen}.
\end{proof}

\begin{cor} \label{Xadjoints}
We have $X^-_\chii(-1) \cong \text{E}^1(\X^+_\theta)$ and $\text{E}^1(X^-_\chii)=\X^+_\theta(-1).$ 
\end{cor}
\begin{proof}
In the spectral sequence of Proposition \ref{spec}, we have $E_2^{3,0}=E^{2,0}_2=E_2^{0,2}=0$ ($E_2^{3,0}=0$ and $E^{2,0}_2=0$  are in \cite[Corollary 4]{jannsen} and \cite[Lemma 5]{jannsen} respectively, and $E^{0,2}_2=0$ because $H^2(G_\infty, \QpZp)=0$ by the abelian case of Leopoldt's conjecture). This gives an isomorphism $H_\infty^2 \cong \text{E}^1(H^1(G_\infty,\QpZp)^\vee)=\text{E}^1(\X)$. Thus,
$$
X^-_\chii(-1) = X^-(-1)_{\theta^{-1}} \cong (H^2_\infty)_{\theta^{-1}} \cong \text{E}^1(\X)_{\theta^{-1}} = \text{E}^1(\X^+_\theta).
$$

We now have $\text{E}^1(X^{-}_\chii(-1))=\text{E}^1(\text{E}^1(\X^+_\theta))=T_1(\X^+_\theta)/T_0(\X_\theta^+)$ using item \ref{EE} of Proposition \ref{e and t}. The result now follows from the fact that $\X^+$ is torsion and has no non-zero finite submodules.
\end{proof}

\subsection{} Let $\xi_\chii \in \Lam$ be a characteristic power series of $X^-_\chii$ (in fact, we make a special choice -- see Remark \ref{choice of xi}).  We require some lemmas.

\begin{lem}\label{dia}
There exists a natural commutive diagram with exact rows:
\[\xymatrix{
0 \ar[r] & U_\chii^-(-1) \ar[r] \ar[d] & \X_\chii^-(-1) \ar[r] \ar[d]^{\nu(-1)} & X^-_\chii(-1) \ar[r] & 0 \\
0 \ar[r] & \Lam^\# \ar[r] & \Lam^\# \ar[r] &  \Lam^\#/\tau \xi_\chii \ar[r] & 0}\]
\end{lem}

Thus the map $\nu(-1)$ induces a map $\bnu: \Xc(-1) \to \Lam^\#/ \tau \xi_\chii$. By Corollary \ref{Xadjoints}, we have $\text{E}^1(\bnu): \Lam/\iota \tau \xi_\chii \to \X^+_\theta$. The next lemma describes the cokernel of $\text{E}^1(\bnu)$.

\begin{lem}\label{extseq}
There is a natural exact sequence:
$$\xymatrix{
 \Lam/\iota \tau \xi_\chii \ar[r]^-{E^1(\bnu)} & \X_\theta^+ \ar[r] & X_\theta^+ \ar[r] & 0.
}$$
\end{lem}

\begin{cor}\label{injective}
The following are equivalent:
\begin{abcs}
\item $X_\theta^+$ is finite.
\item $E^1(\bnu)$ is injective.
\end{abcs}
\end{cor}
\begin{proof}
Since the characteristic ideal of $\X_\theta^+$ is $(\iota \tau \xi_\chii)$, the characteristic ideals of $X_\theta^+$ and $\ker(\text{E}^1(\bnu))$ are equal. Thus  $X^+_\theta$ is finite if and only if  $\ker(\text{E}^1(\bnu))$ is finite. But $\ker(\text{E}^1(\bnu))$ injects into $\Lam/\iota \tau \xi_\chii$, which has no non-zero finite submodules, so $\ker(\text{E}^1(\bnu))$ is finite if and only if $\ker(\text{E}^1(\bnu))=0$. 
\end{proof}

We leave the proofs of Lemmas \ref{dia} and \ref{extseq} to the next section. Using these lemmas, we get the following result.

\begin{prop} \label{cokers}
\begin{abcs}
\item \label{ker} We have $E^1(coker(\bnu)) = ker(E^1(\bnu)).$
\item Assume $X_\theta^+$ is finite. Then we have $\text{coker}(\text{E}^1(\bnu))=\text{coker}(\bnu)^\vee$.
\end{abcs}
\end{prop}
\begin{proof}

For ease of notation, we let $\mathcal{B} \subset \Lam^\#/\tau \xi_\chii$ be the image of $\bnu$, $\mathcal{K}=\text{ker}(\bnu)$, and $\mathcal{C}=\text{coker}(\bnu)$. 

Applying the Iwasawa adjoint to the sequence
$$
0 \to \mathcal{K} \to X_\chii^-(-1) \to \mathcal{B} \to 0
$$
we get
\[\xymatrix{
0  \ar[r] & \text{E}^1(\mathcal{B}) \ar[r] & \text{E}^1(X_\chii^-(-1)) \ar[r] & \text{E}^1(\mathcal{K}) \ar[r] & 0
}\]
noting that $\text{E}^0(\mathcal{K})=\text{E}^2(\mathcal{B})=0$.

Applying the Iwasawa adjoint to the sequence
$$
0 \to \mathcal{B} \to \Lam^\#/\tau \xi_\chii \to \mathcal{C} \to 0
$$
we get
$$\xymatrix{
0 \ar[r] & \text{E}^1(\mathcal{C}) \ar[r] & \text{E}^1(\Lam^\#/\tau \xi_\chii) \ar[r] & \text{E}^1(\mathcal{B}) \ar[r] & \text{E}^2(\mathcal{C}) \ar[r] & 0,
}$$
noting that $\text{E}^0(\mathcal{B})=\text{E}^2(\Lam^\#/\xi_\chii)=0$. 

1) Since the composition $\text{E}^1(\Lam^\#/\tau \xi_\chii) \to \text{E}^1(\mathcal{B}) \to \text{E}^1(X_\chii^-(-1))$ is $\text{E}^1(\bnu)$, we see that $\text{ker}(\text{E}^1(\bnu)) = \text{E}^1(\mathcal{C})$.

2) By Corollary \ref{injective} and the above part (1), we see $\text{E}^1(\mathcal{C})=0$. By Proposition \ref{e and t}, this implies that $\mathcal{C}$ is finite and $\text{E}^2(\mathcal{C})=\mathcal{C}^\vee$.

Putting this together, we get a commutative diagram with exact rows
\
$$\xymatrix{
0 \ar[r] & \Lam/\iota \tau \xi_\chii \ar[r]^-{\text{E}^1(\bnu)} \ar[d] & \X_\theta^+ \ar[r] \ar@{=}[d] & X_\theta^+ \ar[r] \ar[d] & 0 \\
0  \ar[r] & \text{E}^1(\mathcal{B}) \ar[r]^-{\text{E}^1(\bnu)} & \text{E}^1(X_\chii^-(-1)) \ar[r] & \text{E}^1(\mathcal{K}) \ar[r] & 0
}$$
which implies that $\text{E}^1(\mathcal{K})$ is a quotient of $X_\theta^+$, and is thus finite. By Proposition \ref{e and t} we have that $T_1(\mathcal{K})$ is finite, but $\mathcal{K} \subset X_\chii^-(-1)$ is torsion, so $\mathcal{K}$ is finite. But $X_\chii^-(-1)=\text{E}^1(\X_\theta^+)$, so $X_\chii^-(-1)$ has no finite submodules -- thus $\mathcal{K}=0$.

Finally, the sequence
$$\xymatrix@1{
0 \ar[r] & X_\chii^-(-1) \ar[r]^{\bnu} & \Lam^\#/\tau \xi_\chii \ar[r] & \mathcal{C} \ar[r] & 0}
$$
yields
$$\xymatrix@1{
0 \ar[r] & \text{E}^1(\Lam^\#/\tau \xi_\chii) \ar[r]^{\text{E}^1(\bnu)} & \text{E}^1(X_\chii^-(-1)) \ar[r] & \text{E}^2(\mathcal{C}) \ar[r] & 0}
$$
so $\mathcal{C}^\vee = \text{E}^2(\mathcal{C}) = \text{coker}(\text{E}^1(\bnu)) = X_\theta^+$.
\end{proof}

\subsection{} We now explain how Proposition \ref{cokers} implies our main result, Theorem \ref{main}.

\begin{thm}
Suppose $X^-_\chi \ne 0$. If $\fH$ is Gorenstein, then $X_\theta^+=0$.
\end{thm}
\begin{proof}
Recall the diagram

\begin{equation*}\tag{*}
\xymatrix{
\X^-_\chii \otimes X_\chi^- \ar[d]^{\nu' \otimes \U} \ar[r]^-{\Th \otimes \U} & \Hom_\h (\fR,\fQ) \otimes \Hom_\h (\fQ,\fP) \ar[d] \\
 (\Lambda/\xi)^\#(1) \otimes \Hom_\h (\fQ,\fP)  \ar[r] & \fP(1).
}
\end{equation*}
This is commutative by Proposition \ref{comm}. Since $\fH$ is Gorenstein, $\U$ is an isomorphism (\cite{cong}, cf. \cite[Proposition 4.10]{sharifi}), and Proposition \ref{surj} implies that the map $\X^-_\chii \otimes X^-_\chi \to \fP(1)$ is surjective. Thus $\nu' \otimes \Upsilon$ is surjective, and, since $X^-_\chi \ne 0$, $\nu'$ is surjective.

By Nakayama's lemma, $\nu': \X_\chii^- \to (\Lam/\xi)^\#(1)$ is surjective if and only if $\nu: \X_\chii^- \to \Lam^\#(1)$ is surjective (note that $\xi \notin \Lam^\times$ because $X^-_\chi \ne 0$, cf. Remark \ref{mainrem}). If $\nu$ is surjective, then $\bnu$ is surjective.

We now have that $\bnu$ is surjective. Then by Proposition \ref{cokers} part (1), we have that $\text{E}^1(\bnu)$ is injective, and so $X^+_\theta$ is finite by Corollary \ref{injective}. We may now apply Proposition \ref{cokers} part (2), which implies that $\text{coker}(\text{E}^1(\bnu))=0$. But $\text{coker}(\text{E}^1(\bnu))=X_\theta^+$ by Lemma \ref{extseq}.
\end{proof}

This completes the proof of part (1) of Theorem \ref{main}. 

\begin{thm}\label{converse}
If $X^+_\theta=0$ and $\X_\theta^+ \ne 0$, then $\nu'$ is surjective. If, in addition, $\U$ is surjective, then $\Phi$ is surjective and, in particular, $\fH$ is Gorenstein.
\end{thm}
\begin{proof}
It is immediate from Lemma \ref{extseq} and Proposition \ref{cokers} that $X^+_\theta=0$ implies that $\bnu$ is surjective. From Corollary \ref{Xadjoints} we see that the assumption $\X_\theta^+ \ne 0$ is equivalent to $\xi_\chii \notin \Lam^\times$. By Nakayama's lemma, we have that $\nu$ is surjective, and thus that $\nu'$ is surjective. The second statement now follows from the commutative diagram (*) and Proposition \ref{surj}.
\end{proof}

For the proof of part (2) of Theorem \ref{main}, note that if $\X^+_\theta=0$, then $\fH$ is Gorenstein by Otha's Theorem \ref{ohtaresult}. Otherwise $\X^+_\theta \ne 0$, and we may apply Theorem \ref{converse}. This completes the proof of Theorem \ref{main}.

\section{Coleman power series and explicit reciprocity}

The purpose of this section is to prove Lemmas \ref{dia} and  \ref{extseq}. The motivation for the method of proof is that the map $\text{E}^1(\bnu)$ should coincide with the composite
$$
\Lam/\iota \tau \xi_\chii \to U_\theta^+/C_\theta^+ \to \X_\theta^+
$$
where $C$ is the group of cyclotomic units, $\Lam/\iota \tau \xi_\chii \to U_\theta^+/C_\theta^+$ is the isomorphism of Iwasawa's Theorem defined by Coleman power series, and $U_\theta^+/C_\theta^+ \to \X_\theta^+$ is the natural map from the sequence (\ref{euxx}). We first review the theory of Coleman power series and the relation to explicit reciprocity.

\subsection{} Let $\fO=\Z[\zeta_N] \otimes \Z_p$. Let $Fr \in \Gal(\Q(\zeta_N)/\Q)$ be the Frobenius at $p$ and $\phi \in \text{End}(\fO[[T]])$ be defined as $Fr$ on $\fO$ and $\phi(1-T)=(1-T)^p$. Let $a \in \Z_{p}^\times$ act on $\fO[[T]]$ by $a.(1-T)=(1-T)^{a}$  and extend this to an action of $\fO[[\Z_{p}^\times]]$. Let $\epsilon : \fO[[\Z_{p}^\times]] \to \fO[[\Z_{p}^\times]]$ be the map induced by $a \mapsto -a$ for $ a \in \Z_{p}^\times$.

 We have the Coleman map
$$
Col:U \to \fO[[\Z_{p}^\times]]
$$
defined for $u=(u_n)\in U$ to be the unique element such that $(1-\phi/p)log(f_u)=Col(u).(1-T)$, where $f_u$ is the unique element of $\fO[[T]]$ such that $f_u(1-\zeta_{p^{n+1}})=Fr^n(u_n)$ for all $n$. It induces isomorphisms
$$
Col: U^+_\theta \to \Lam \ , \ \iota \tau Col: U^-_\chii \to \Lam^\#(1).
$$

The following theorem is known as Iwasawa's Theorem. For a proof, see \cite[Lemma 2.11 and Theorem 2.13]{greither}.

\begin{thm}\label{iwasawa's theorem}
The element $\cyc$ is a generator of $C^+_\theta$ as a $\Lambda$-module.  We have  that $\iota \tau(Col(C^+_\theta))=\xi_\chii\Lam$
\end{thm}

\begin{rem}\label{choice of xi}
In light of this theorem, we see that $\iota \tau (Col(\cyc))$ is a characteristic power series for $X_\chii^-$, so we may choose $\xi_\chii=\iota \tau (Col(\cyc))$.
\end{rem}

The following explicit reciprocity law is essentially due to Coleman (\cite{coleman}).

\begin{thm}\label{explicit res}
The following diagram is commutative:

\[\xymatrixcolsep{4pc}\xymatrix{
U^+_\theta \times U^-_\chii   \ar[d]_{Col}^{ \iota \tau \epsilon Col} \ar[r]^-{( \ , \ )_{Kum}}   & \Lam^\#(1) \ar[d] \\
\Lam \times \Lam^\#(1)  \ar[r] & \Lam^\#(1). \\} \]
Here the bottom horizontal map is the product.
\end{thm}
\begin{proof}
This is proven in \cite[Th\'{e}or\`{e}me 4.3.2]{p-r}. To be more precise, let $V = \fO_{\Q_p(\zeta_N)}[[\Z_{p}^\times]].(1-T)$, $Z=\varprojlim \fO_{\Q_p(\zeta_{Np^r})}^\times$.  The commutativity of the 
diagram 
\begin{equation}\tag{P-R}
\xymatrixcolsep{4pc}\xymatrix{
Z \times Z \ar[r]^-{( \ , \ )_{Kum}} & \Z_p[[\Z_{p,N}^\times]] \ar[d] \\
V \times V \ar[u]^-{\Omega^{(\zeta_{p^r})}_{\Q_p(1),1}}_-{\Omega^{(\zeta_{p^r}^{-1})}_{\Q_p(1),1}} \ar[r] & \fO_{\Q_p(\zeta_N)}[[T]]
}
\end{equation}
is proven, where the unlabeled horizontal map is given by $$(f ,g) \mapsto f*tDg.
$$
Here, $*$, $t$, and $D$ correspond via the isomorphism $V \cong \fO_{\Q_p(\zeta_N)}[[\Z_{p}^\times]]$ to the product, $\iota$, and $\tau$, respectively. 
 
As in \cite[section 4.1]{p-r}, $(\Omega^{\varepsilon}_{\Q_p(1),1})^{-1}$ is $Col$ with respect to the system of $p$-power roots of unity $\varepsilon$.
 
Theorem \ref{explicit res} now follows by using the commutativity of the diagram (P-R) for each direct factor of $U$, $\fO$, etc., and taking the $\theta$-eigenspace (cf. Lemma \ref{properties of kummer pairing}).
 
Note that $\epsilon$ appears because, in (P-R), the two Coleman maps are taken with respect to inverse systems of roots of unity: if, for some $u \in U$, $F \in \fO[[T]]$ satisfies $F(1-\zeta_{p^{n+1}})=u_n$ and $G \in \fO[[T]]$ satisfies $G(1-\zeta_{p^{n+1}}^{-1})=u_n$, then, by uniqueness, $G(T)=F(1-1/(1-T))$, and consequently $G=\epsilon F$.
\end{proof}

\subsection{} We complete the proof of Lemmas \ref{dia} and \ref{extseq}.
\begin{proof}[Proof of Lemma \ref{dia}]
We wish to show the commutativity of the square
$$\xymatrix{
U_\chii^- \ar[r] \ar[d]^{\iota \tau \epsilon Col}_\wr & \X_\chii^- \ar[d]^\nu \\
\Lam^\#(1) \ar[r]^-{\xi_\chii} & \Lam^\#(1). \\
}$$
That is, for $u \in U^-_\chii$, we wish to show that $$\nu(u)=\iota \tau ( \xi_\chii) \iota \tau \epsilon Col(u)=\iota \tau (\xi_\chii \epsilon Col(u)).$$
By Lemma \ref{properties of kummer pairing} we have $\nu(u)=(\cyc,u)_{Kum}$. By Theorem \ref{explicit res}, we have $$(\cyc,u)_{Kum}=Col(\cyc)\iota \tau (\epsilon (Col(u))))=\iota \tau (\iota \tau Col(\cyc)) \epsilon (Col(u))).$$
We now obtain the result by applying Theorem \ref{iwasawa's theorem}

\end{proof}
\begin{proof}[Proof of Lemma \ref{extseq}] 
When we apply the Iwasawa adjoint to the complex from Lemma \ref{dia} to obtain a commutative diagram with exact rows

\[\xymatrix{
 \text{E}^0(U_\chii^-(-1)) \ar[r] & \text{E}^1(X_\chii^-(-1)) \ar[r] & \text{E}^1(\X_\chii^-(-1)) \ar[r] & 0\\
\text{E}^0(\Lam^\#) \ar[u]^\wr \ar[r] & \text{E}^1(\Lam^\#/\tau \xi_\chii) \ar[u]^{\text{E}^1(\bnu)} \ar[r] & 0. \ar[u] \\}\]

Replacing the adjoints with their isomorphic images, computed in section 4, we obtain

\[\xymatrix{
U_\theta^+ \ar[r] & \X_\theta^+ \ar[r] & X_\theta^+ \ar[r] & 0 \\
 \Lam \ar[u]^\wr \ar[r] & \Lam/\iota \tau \xi_\chii \ar[u]^{E^1(\bnu)} \ar[r] & 0 \ar[u] \\}\]

We can now read off the desired exact sequence.

\end{proof}

Department of Mathematics, Eckhart Hall, University of Chicago, Chicago, Illinois 60637

email: pwake@math.uchicago.edu
\end{document}